\documentclass{amsart}
\usepackage{amsmath,amsthm,amssymb,IMjournal}

\usepackage{adjustbox}
\begin{document}


\title{Jacobson radicals of Ore extensions}

\author{\|Jooyoung |Shin |\ORC{orcid number}}



\abstract 
 Let $R$ be a ring, $\sigma$ be an automorphism of $R$, and $D$ be a $\sigma$-derivation on $R$ with $D\sigma=\sigma D$. We will show that $J(R[x;\sigma,D])\subseteq A[x;\sigma,D]$ where $A$ is a collection of all leading coefficients of polynomials in $J(R[x;\sigma,D])$. 
 
 In addition, we will prove that $J(R[x;\sigma,D])\cap R$ is nil if $\sigma$ is locally torsion and one of the following conditions is given: (1) $R$ is a PI-ring and $D$ is $q$-skew, (2) $R$ is an algebra over a field of characteristic $p>0$ and $D$ is a locally nilpotent derivation such that $\sigma D=D\sigma$. This answers partially an open question by Greenfeld, Smoktunowicz and Ziembowski. 
\endabstract

\keywords
   Jacobson radical; Ore extension; locally torsion automorphism; polynomial identity rings; locally nilpotent derivation
\endkeywords

\subjclass
16N20, 16N40
\endsubjclass

\section{Introduction}

Let $R$ be a ring and $\sigma$ be an automorphism of $R$. An additive map $D: R \to R$ that satisfies $D(ab)=\sigma(a)D(b)+D(a)b$ for all $a,b\in R$ is called a $\sigma$-derivation of $R$. If for every $a\in R$, there exists an integer $n=n_a>0$ such that $\sigma^n(a)=a$, then $\sigma$ is called locally torsion. If for every $a\in R$, there is an integer $m=m_a>0$ such that $D^m(a)=0$, then $D$ is called locally nilpotent.

An Ore extension $R[x;\sigma,D]$ is a ring of polynomials of form $a_nx^n+\dots +a_1x+a_0$ with $n\geq 0$ and $a_0,\dots ,a_n\in R$ with standard addition and multiplication given by $xa=\sigma(a)x+D(a)$ for all $a\in R$. There are two important special cases of Ore extensions. When $\sigma=1$, the Ore extension is called a differential polynomial ring or a skew polynomial ring of derivation type and denoted by $R[x;D]$. When $D=0$, the Ore extension is called a skew polynomial ring of automorphism type and denoted by $R[x;\sigma]$.

Let $J(R)$ be the Jacobson radical of $R$. It is well-known that every element $r_1$ in $J(R)$ is quasi-regular which means that there exists an element $r_2$ in $J(R)$ such that 
\begin{center}
    $0=r_1+r_2+r_1r_2=r_1+r_2+r_2r_1$.
\end{center}

One famous open question in ring theory is the Koethe conjecture asking whether every one-sided nil ideal is contained in a two-sided nil ideal \cite{ko}. In 1972, Krempa showed that Koethe conjecture is equivalent to whether a polynomial ring over a nil ring is Jacobson radical \cite{P11}. In 1956, Amitsur proved that the Jacobson radical of a polynomial ring, $J(R[x])$, is equal to $(J(R[x])\cap R)[x]$ and $J(R[x])\cap R$ is a nil ideal of $R$ \cite{P7}. 
In 1983, Ferrero, Kishimoto, and Motose proved that $J(R[x;D])=(J(R[x;D])\cap R)[x;D]$ \cite{P5}. They showed that if $R$ is a commutative ring, then $J(R[x;D])\cap R$ is nil \cite{P5}. 
In 1980, Bedi and Ram proved that the Jacobson radical of a skew polynomial ring of automorphism type, $J(R[x;\sigma])$, is equal to $K\cap J(R)+K_0$ where $K=\{r\in R : rx\in J(R[x;\sigma])\}$ and $K_0=\{\sum_{i\geq 1}r_ix^i:r\in K\}$ \cite{P13}. In particular, they showed that if $\sigma$ is locally torsion, then $K$ is nil and so $J(R[x;\sigma])=K[x;\sigma]$ \cite{P13}.
However, Amitsur's theorem does not hold for general Ore extensions. Bedi and Ram showed a counter-example \cite{P13}. In section 2, we will extend their result to Ore extension and show that if $D\sigma=\sigma D$, then $J(R[x;\sigma,D])\subseteq A[x;\sigma,D]$ where $A$ is a collection of all leading coefficients of polynomials in $J(R[x;\sigma,D])$.

In 2016, Madill proved that if $R$ is a polynomial identity ring, then $J(R[x;D])\cap R$ is nil \cite{P4}. In 2017, Smoktunowicz proved that if $R$ is an algebra over a field $F$ of characteristic $p>0$ and $d$ is a locally nilpotent derivation on $R$, then $J(R[x;D])\cap R$ is a nil ideal of $R$ \cite{P8}. We may wonder whether such results can be gained on Ore extensions. In 2019, Greenfeld, Smoktunowicz and Ziembowski posed the question below. In section 3, we will give partial answers to the question.

\begin{question}\cite[Question 6.17]{P2}\label{q1}
    Consider $J(R[x;\sigma,D])\cap R$. Is it nil if we assume that $D$ is locally nilpotent?
What if we assume that $\sigma$ is locally torsion?
\end{question}

\section{Jacobson radicals of Ore extensions}
Let $R$ be a ring, $\sigma$ be an automorphism of $R$, and $D$ be a $\sigma$-derivation on $R$.
Suppose that $R$ does not have identity $1$. Let $F$ be a field. Consider $R^\#=R\oplus F$ which is the Dorroh extension of $R$ with componentwise addition and multiplication defined by $(r_1,m_1)(r_2,m_2)=(r_1r_2+m_2r_1+m_1r_2, m_1m_2)$ for $r_1, r_2\in R$ and $m_1,m_2\in F$. 
Then $R^\#$ has identity $(0,1)$ and $R$ is a $F$-subalgebra of $R^\#$.
Define $\sigma^\#((r,m))=(\sigma(r),m)$ and $D^\#((r,m))=(D(r),0)$ for any $(r,m)\in R\oplus F$. We can check that $\sigma^\#$ and $D^\#$ are an automorphism and $\sigma^\#$-derivation on $R\oplus F$. Since $R^\#[x;\sigma^\#,D^\#]/R[x;\sigma,D]\cong F[x]$ and $F$ is a field, so $J(R^\#[x;\sigma^\#,D^\#])=J(R[x;\sigma,D])$.

If there is a central and invertible element $q\in R^\#$ such that $D\sigma =q\sigma D$, then $D$ is called $q$-skew derivation. We can find following useful formulas in \cite{ct} and \cite{P15}.

\begin{proposition}\label{p1}
    Let $\sigma$ be an endomorphism of $R$, and $D$ be a $q$-skew $\sigma$-derivation on $R$. Then for any $r\in R$ and non-negative integer $k$,
    \begin{align*}
        x^kr=\sum_{i=0}^k \binom{k}{i}_q \sigma^{i}D^{k-i}(r)x^i
    \end{align*}
    where $\binom{k}{j}_q$ is the evaluation at $t=q$ of a polynomial function 
    \begin{align*}
        \binom{k}{j}_t=\frac{(t^k-1)(t^{k-1}-1)\dots(t^{k-j+1}-1)}{(t^j-1)(t^{j-1}-1)\dots(t-1)}.
    \end{align*}
    In particular, for any positive integer $m$, if $q$ is a primitive $m$-th root of unity, then $\binom{km}{jm}_q=\binom{k}{j}$.
\end{proposition}

Especially, when $q=1$, the following proposition is gained.

\begin{proposition}\label{p2}
    Let $\sigma$ be an endomorphism of $R$, and $D$ be a $\sigma$-derivation on $R$ with $D\sigma=\sigma D$. Then for any $r\in R$ and non-negative integer $k$,
    \begin{align*}
        x^kr=\sum_{i=0}^k \binom{k}{i}\sigma^{i}D^{k-i}(r)x^i
    \end{align*}
\end{proposition}

\begin{lemma}\label{l1}
    Let $\sigma$ be an automorphism on $R$ and $D$ be a $q$-skew $\sigma$-derivation on $R$. Then $J(R[x;\sigma,D])\cap R$ is closed under $\sigma$.
\end{lemma}

\begin{proof}
    Define a map $\sigma^*: R[x;\sigma,D] \rightarrow R[x;\sigma, D]$ by $\sigma^*(\sum_{i=0}^{n}a_ix^i)=\sum_{i=0}^nq^{-i}\sigma(a_i)x^i$ for any $a_i\in R$. Since $\sigma$ is an automorphism and $q$ is invertible and central in $R^\#$, $\sigma^*$ is well-defined and bijective. Now, we claim that $\sigma^*$ is an automorphism. Let $f(x)=\sum_{i=0}^na_ix^i$ and $g(x)=\sum_{i=0}^mb_ix^i$ for $n\leq m$ be polynomials in $R[x;\sigma,D]$.
    \begin{align*}
        \sigma^*(f(x)+g(x))&=\sigma^*(\sum_{i=0}^n(a_i+b_i)x^i+\sum_{i=n+1}^mb_ix^i)\\
        &=\sum_{i=0}^nq^{-i}\sigma(a_i+b_i)x^i+\sum_{i=n+1}^mq^{-i}\sigma(b_i)x^i\\
        &=\sum_{i=0}^n(q^{-i}\sigma(a_i)+q^{-i}\sigma(b_i))x^i+\sum_{i=n+1}^mq^{-i}\sigma(b_i)x^i\\
        &=\sum_{i=0}^nq^{-i}\sigma(a_i)x^i + \sum_{i=0}^mq^{-i}\sigma(b_i)x^i\\
        &=\sigma^*(f(x))+\sigma^*(g(x)).
    \end{align*}
    To show that $\sigma^*(f(x)g(x))=\sigma^*(f(x))\sigma^*(g(x))$, it suffices to show that $\sigma^*(ax^nbx^m)=\sigma^*(ax^n)\sigma^*(bx^m)$. Since $D$ is $q$-skew, by Proposition \ref{p1}
    \begin{align*}
        \sigma^*(ax^nbx^m)&=\sigma^*(\sum_{i=0}^n\binom{n}{i}_qa\sigma^iD^{n-i}(b)x^{i+m})\\
        &=\sum_{i=0}^n\binom{n}{i}_q q^{-(i+m)}\sigma (a\sigma^iD^{n-i}(b))x^{i+m}\\
        &=\sum_{i=0}^n\binom{n}{i}_q q^{-(i+m)}\sigma (a)\sigma^{i+1}D^{n-i}(b)x^{i+m}\\
        &=\sum_{i=0}^n\binom{n}{i}_q q^{-(i+m)}q^{-(n-i)}\sigma (a)\sigma^iD^{n-i}\sigma(b)x^{i+m}\\
        &=q^{-n-m}\sigma (a)\left(\sum_{i=0}^n\binom{n}{i}_q \sigma^iD^{n-i}\sigma(b)x^{i+m}\right)\\
        &=(q^{-n}\sigma (a)x^n)(q^{-m}\sigma(b)x^m)\\
        &=\sigma^*(ax^n)\sigma^*(bx^m).
    \end{align*}
    Thus, $\sigma^*$ is an automorphism on $R[x;\sigma,D]$.
    We know that Jacobson radicals are closed under isomorphisms. That means that $\sigma^*(J(R[x;\sigma,D]))\subseteq J(R[x;\sigma,D])$.
    Then for any $a\in J(R[x;\sigma,D])\cap R$,
    we have that $\sigma(a)=\sigma^*(a)\in J(R[x;\sigma,D])\cap R$. 

\end{proof}

\begin{lemma}\label{l2}
    Let $R$ be a ring, $\sigma$ be an automorphism of $R$, and $D$ be a $q$-skew $\sigma$-derivation on $R$.
    Then, $J(R[x;\sigma, D])\cap R$ is a $D$-ideal of $R$.
\end{lemma}

\begin{proof}
    We can asuume that $R$ has an identity.
    Since $J(R[x;\sigma,D])$ is an ideal of $R[x;\sigma,D]$, $J(R[x;\sigma, D])\cap R$  is an ideal of $R$. Take $a\in J(R[x;\sigma, D])\cap R$. By Lemma \ref{l1}, $\sigma(a)\in J(R[x;\sigma, D])\cap R$.
    Since $xa$ and $\sigma(a)x$ are in $J(R[x;\sigma, D])$,
    we get that $D(a)=xa-\sigma(a)x\in J(R[x;\sigma, D])$.
    Since $D(a)\in R$,  thus $D(a)\in J(R[x;\sigma, D])\cap R$.
\end{proof}

From Lemma \ref{l1} and \ref{l2},  we can define $(J(R[x;\sigma,D])\cap R)[x;\sigma,D]$. It is an ideal of $R[x;\sigma,D]$. However, in general, the equality
\begin{align*}
    J(R[x;\sigma,D])=(J(R[x;\sigma,D])\cap R)[x;\sigma,D]
\end{align*}
is false. One can find the following counterexample in \cite[Example 3.5]{P13}.

\begin{example}\cite[Example 3.5]{P13}\label{ex1}
    Let $R=\oplus_{i\in \mathbb{Z}}S_i$ where $S_i=S$ is an arbitrary ring which is not Jacobson radical. Let $\sigma:R\rightarrow R$ be an automorphism such that $\sigma(\sum_{i\in \mathbb{Z}}a_i)=\sum_{i\in \mathbb{Z}}a_{i-1}$. Consider $R[x;\sigma]$. 
    Let $K=\{r\in R : rx\in J(R[x;\sigma])\}$. 
    Take any $a=\sum_{i\in \mathbb{Z}}a_i\in R$.
    If $a_i=0$ for all $i\notin [m,n]$, then $(axR[x;\sigma])^{n-m+2}=0$. Thus, $K=R$.
    Since $J(R[x;\sigma])=K\cap J(R)+K_0$ where $K_0=\{\sum_{i\geq 1}r_ix^i:r\in K\}$, so $J(R[x;\sigma])=J(R)+K_0$. Since $J(R)=\oplus_{i\in \mathbb{Z}}J(S_i)$, we obtain that $J(R[x;\sigma])\neq J(R)[x;\sigma]$.
\end{example}

Now, one may ask the following question. 

\begin{question}\label{q2}
    When is $J(R[x;\sigma,D])$ equal to $(J(R[x;\sigma,D])\cap R)[x;\sigma,D]$? 
\end{question}

We will prove that $J(R[x;\sigma,D])\subseteq A[x;\sigma,D]$ where $A$ is a collection of all leading coefficients of polynomials in $J(R[x;\sigma,D])$.

\begin{lemma}\label{l3}
    Let $R$ be a ring, $\sigma$ be an automorphism of $R$, and $D$ be a $\sigma$-derivation of $R$ such that $D\sigma=\sigma D$. 
    Let $A$ be a collection of all leading coefficients of polynomials in $J(R[x;\sigma,D])$.
    Then $A$ is a $(\sigma,D)$-ideal of $R$.
\end{lemma}

\begin{proof}
    We can assume that $R$ has an identity.
    First, $A$ is closed under addition. 
    Take any $a_n, b_m\in A$.
    Then there are polynomials $f(x)=a_nx^n+\dots+a_1x+a_0$ and $g(x)=b_mx^m+\dots+b_1x+b_0$ in $J(R[x;\sigma,D])$.
    Let $n\leq m$.
    Since $J(R[x;\sigma,D])$ is an ideal of $R[x;\sigma,D]$, $f(x)x^{m-n}+g(x) \in J(R[x;\sigma,D])$.
    Then we can see that $a_n+b_m$ is the leading coefficient of $f(x)x^{m-n}+g(x)$ and so $a_n+b_m$ is an element of $A$.

    Second, $A$ is closed under multiplication and it is a left ideal of $R$.
    Let $a_n\in A$. 
    We can find a polynomial $f(x)=a_nx^n+\dots+a_1x+a_0$ in $J(R[x;\sigma,D])$.
    Since $J(R[x;\sigma,D])$ is an ideal of $R[x;\sigma,D]$, $rf(x)\in R[x;\sigma,D]$ for every $r\in R$.
    So, its leading coefficient $ra_n$ is in $A$.

    Third, $A$ is a right ideal.
    Let $a_n\in A$ and $f(x)=a_nx^n+\dots+a_1x+a_0$ is a polynomial in $J(R[x;\sigma,D])$.
    Take any $r\in R$.
    Since $\sigma$ is an automorphism, there is an element $e\in R$ such that $\sigma^{n}(e)=r$.
    Then $f(x)e=(a_nx^n+\dots+a_1x+a_0)e$ has the leading coefficient $a_n\sigma^n(e)=a_nr$.
    So, $a_nr\in A$.

    Now, we will show that $A$ is $(\sigma,D)$-stable.
    Let $a_n\in A$ and $f(x)=a_nx^n+\dots+a_1x+a_0$ is a polynomial in $J(R[x;\sigma,D])$.
    Define a map $\sigma^*: R[x;\sigma,D] \rightarrow R[x;\sigma, D]$ by $\sigma^*(\sum_{i=0}^{n}a_ix^i)=\sum_{i=0}^n\sigma(a_i)x^i$ for any $a_i\in R$. 
    Since $\sigma$ is an automorphism, $\sigma^*$ is an well-defined automorphism by the proof of Lemma \ref{l1}.
    So, $\sigma^*(f(x))=\sigma(a_n)x^n+\dots+\sigma(a_1)x+\sigma(a_0)$ is in $J(R([x;\sigma,D]))$ and so $\sigma(a_n)\in A$.

    Let $a_n\in A$ and $f(x)=a_nx^n+\dots+a_1x+a_0\in J(R[x;\sigma,D])$.
    Since $J(R[x;\sigma,D])$ is an ideal, $xf(x)\in J(R[x;\sigma,D])$.
    \begin{align*}
        xf(x)&=xa_nx^n+xa_{n-1}x^{n-1}+\dots+xa_1x+xa_0\\
        &=(\sigma(a_n)x^{n+1}+D(a_n)x^n)+(\sigma(a_{n-1})x^{n}+D(a_{n-1})x^{n-1})\\
        &\hspace{2.7in}+\dots+(\sigma(a_{0})x+D(a_{0}))\\
        &=\sigma(a_n)x^{n+1}+(D(a_n)+\sigma(a_{n-1}))x^n+\dots+D(a_0)
    \end{align*}
    From the previous paragraph, $\sigma^*(f(x))x=\sigma(a_n)x^{n+1}+\dots+\sigma(a_1)x^2+\sigma(a_0)x$ is in $J(R([x;\sigma,D]))$.
    So, $xf(x)-\sigma^*(f(x))x\in J(R([x;\sigma,D]))$.
    Since $xf(x)-\sigma^*(f(x))=D(a_n)x^n+\dots+D(a_0)$, its leading coefficient $D(a_n)$ is in $A$.    
\end{proof}

\begin{theorem}\label{t1}
    Let $R$ be a ring, $\sigma$ be an automorphism of $R$, and $D$ be a derivation of $R$ such that $D\sigma=\sigma D$. 
    Let $A$ be a collection of all leading coefficients of polynomials in $J(R[x;\sigma,D])$.
    Then $J(R[x;\sigma,D])\subseteq A[x;\sigma,D]$.
\end{theorem}

\begin{proof}
    The case where $J(R[x;\sigma,D])=0$ is trivial. 
    Suppose that $J(R[x;\sigma,D])\neq 0$.
    Then $A\neq 0$. 
    By Lemma \ref{l3}, $A$ is a $(\sigma,D)$-ideal of $R$.
    Consider $K=(J(R[x;\sigma,D])+A[x;\sigma,D])/A[x;\sigma,D]$.
    Suppose that $K\neq 0$.
    Let $f(x)+A[x;\sigma,D]$ be a nonzero polynomial in $K$.
    Then we can find $p(x)+A[x;\sigma,D]=f(x)+A[x;\sigma,D]$ for some $p(x)\in J(R[x;\sigma,D])$.
    Take any polynomial $g(x)\in R[x;\sigma,D]$.
    Since $p(x)\in J(R[x;\sigma,D])$, $p(x)g(x)\in J(R[x;\sigma,D])$ and there exists $k(x)\in J(R[x;\sigma,D])$ such that $p(x)g(x)+k(x)+p(x)g(x)k(x)=0$.
    So,
    \begin{align*}
    0+A[x;\sigma,D]&=p(x)g(x)+k(x)+p(x)g(x)k(x)+A[x;\sigma,D]\\
        &=(p(x)+A[x;\sigma,D])(g(x)+A[x;\sigma,D])+(k(x)+A[x;\sigma,D])\\
        &\hspace{0.8in}+(p(x)+A[x;\sigma,D])(g(x)+A[x;\sigma,D])(k(x)+A[x;\sigma,D])\\
        &=(f(x)+A[x;\sigma,D])(g(x)+A[x;\sigma,D])+(k(x)+A[x;\sigma,D])\\
        &\hspace{0.8in}+(f(x)+A[x;\sigma,D])(g(x)+A[x;\sigma,D])(k(x)+A[x;\sigma,D])\\
        &=f(x)g(x)+k(x)+f(x)g(x)k(x)+A[x;\sigma,D].
    \end{align*}
    Since $K=(J(R[x;\sigma,D]+A[x;\sigma,D])/A[x;\sigma,D]\cong J(R[x;\sigma,D])/(J(R[x;\sigma,D])\cap A[x;\sigma,D])$, so $f(x)g(x)+k(x)+f(x)g(x)k(x)\in J(R[x;\sigma,D])\cap A[x;\sigma,D]$.
    There exists $l(x)\in J(R[x;\sigma,D])$ such that 
    \begin{align*}
        0&=( f(x)g(x)+k(x)+f(x)g(x)k(x))+l(x)+(f(x)g(x)+k(x)+f(x)g(x)k(x))l(x)\\
        &=f(x)g(x)+k(x)+f(x)g(x)k(x)+l(x)\\
        &\hspace{2in}+f(x)g(x)l(x)+k(x)l(x)+f(x)g(x)k(x)l(x)\\
        &=f(x)g(x)+(k(x)+l(x)+k(x)l(x))+f(x)g(x)(k(x)+l(x)+k(x)l(x)).
    \end{align*}
    Thus, $f(x)g(x)$ is quasi-invertible. 
    Since we took arbitary polynomial $g(x)\in R[x;\sigma,D]$, $f(x)\in J(R[x;\sigma,D])$. 
    Then the leading coefficient of $f(x)$ is in $A$. 
    That contradicts that $f(x)+A[x;\sigma,D]$ is a nonzero polynomial in $K$.
    Hence, $K=0$ and so $J(R[x;\sigma,D])\subseteq A[x;\sigma,D]$.
    
\end{proof}

\section{Ore extensions with locally torsion automorphisms}

In 2011, Bergen and  Grzeszczuk considered the skew power series ring $R[[x;\sigma,D]]$ as all infinite sums in form of $a_0+a_1x+a_2x^2+\dots$ for all $a_0, a_1, a_2, \dots\in R$ \cite{bg}. The skew power series ring is closed under standard addition. However, the multiplication given by $xa=\sigma(a)x+D(a)$ is not necessarily well-defined. In 2019, Greenfeld, Smoktunowicz, and Ziembowski proved that if $D$ is locally nilpotent, the natural multiplication from $R[x;\sigma,D]$ to $R[[x;\sigma,D]]$ is well-defined \cite{P2}.

\begin{lemma}\cite[Theorem 5.3]{P2}\label{l10}
    Let $R$ be a ring, $\sigma$ be an endomorphism, and $D$ be a locally nilpotent $\sigma$-derivation. Then natural multiplication from $R[x;\sigma,D]$ to $R[[x;\sigma,D]]$ is well-defined. 
\end{lemma}

Now, we will investigate when $J(R[x;\sigma, D])\cap R$ is a nil ideal of $R$. Specifically, we will consider Ore extensions $R[x;\sigma,D]$ where $\sigma$ is locally torsion. Throughout this section, the center of $R$ is denoted by $Z(R)$ and the nilradical of $R$ is denoted by $N(R)$.

\begin{lemma}\cite[Theorem 1.6.14]{P1}\label{l11}
    Let $R$ be a PI ring such that $N(R)=0$.
    Then every nonzero ideal of $R$ intersects $Z(R)$ nontrivially.
\end{lemma}

\begin{lemma}\label{l12}
    Let $R$ be a ring. Let $a$ be an element in $R$ such that $a+N(R)\in Z(R/(N(R)))$. If $a^t\in N(R)$ for some positive integer $t$, then $a\in N(R)$.
\end{lemma}

\begin{proof}
    Since $a^t\in N(R)$, $a$ is nilpotent.
    Since $a+N(R)\in Z(R/(N(R)))$, we have that $(a\mathbb{Z}+aR+N(R))/(N(R))$ is a nil ideal of $R/N(R)$ and so $a\mathbb{Z}+aR+N(R)$ is a nil ideal of $R$. Thus, $a\mathbb{Z}+aR+N(R)$ is contained in $N(R)$. 
    That means that $a\in N(R)$.
\end{proof}

\begin{theorem} \label{t3}
    Let $R$ be a PI ring.
    Let $\sigma$ be a locally torsion automorphism of $R$ and $D$ be a $q$-skew $\sigma$-derivation of $R$.
    Then $J(R[x;\sigma,D])\cap R$ is a nil $(\sigma,D)$-ideal of $R$.
\end{theorem}

\begin{proof}
    The approach follows \cite{P4}.
    Let $S=J(R[x;\sigma,D])\cap R$ and let $N=N(R)$.
    We know that $S$ is an $(\sigma,D)$-ideal of $R$ by Lemma \ref{l1} and Lemma \ref{l2}. We will prove that $S$ is nil.
    Consider the quotient ring $L:=(S+N)/N$. 
    Note that $L$ is an ideal of $R/N$, which is a PI ring with the zero upper nilradical.

    If $L=0$, then $S\subseteq N$ and so $S$ is nil.
    Suppose that $L\neq 0$.
    By Lemma \ref{l11}, $L\cap Z(R/N)\neq \{0\}$.
    Take a nonzero element $a+N\in L\cap Z(R/N)$. 
    Then $a\in S$, $a+N\neq N$, and $a+N\in Z(R/N)$.
    Let $R^*$ be a ring $R$ with adjoined identity.
    Since $a\in S\subseteq J(R[x;\sigma,D])$ and $J(R[x;\sigma,D])$ is an ideal of $R^*[x;\sigma,D]$, we have that $ax\in J(R[x;\sigma,D])$ and so there exists $f(x)=\sum_{i=0}^mb_ix^i\in J(R[x;\sigma,D])$ such that
    \begin{equation}\label{e1}
        f(x)+ax+f(x)ax=0.
    \end{equation}
    Then
    \begin{align*}
        0&=(b_mx^m+\dots+b_0)+ax+(b_mx^m+\dots+b_0)ax\\
        &=(b_mx^m+\dots+b_0)+ax+(b_mx^max+\dots+b_0ax).
    \end{align*}

    \noindent By Proposition \ref{p1}, 
    \begin{align*}
        &(b_mx^m+b_{m-1}x^{m-1}+\dots+b_1x+b_0)+ax\\
        &\hspace{1.3in}+(w_{m+1}x^{m+1}+w_{m}x^{m}+\dots+w_nx)=0
    \end{align*}
    where $w_{i}=\sum_{j=i-1}^m\binom{j}{i-1}_qb_j\sigma^{i-1}D^{j-i+1}(a)$ for all $0\leq i\leq m+1$.
    So, $b_0=0$.
    From the $x^{m+1}$ term, 
    \begin{align}\label{e2}
        0=w_{m+1}=\binom{m}{m}_qb_m\sigma^m(a)=b_m\sigma^m(a).
    \end{align}
    From the $x^{i}$ term for $i=2, 3, \dots,m$, 
    \begin{align}\label{e4}
        0=b_{i}+w_{i}=b_{i}+\sum_{j=i-1}^m\binom{j}{i-1}_qb_j\sigma^{i-1}D^{j-i+1}(a).
    \end{align}
    From the $x$ term,
    \begin{equation}\label{e5}
    \begin{split}
        0&=b_1+a+w_1\\
        &=b_1+a+\sum_{j=0}^m\binom{j}{0}_qb_jD^{j}(a)=b_1+a+\sum_{j=1}^mb_jD^{j}(a).
    \end{split}
    \end{equation}

    Since $\sigma$ is locally nilpotent, there is an integer $n>0$ such that $\sigma^n(a)=a$.
    Let $A$ be a set of all products of $a, \sigma(a), \sigma^2(a), \dots, \sigma^{n-1}(a)$.
    We claim that there exists some $r_k\in A$ such that $b_kr_k\in N$ for all $k=1, 2, \dots,m$.
    We will proceed by induction on $k$.
    For $k=m$, the equation (\ref{e2}) shows $b_m\sigma^m(a)=0$. Since $\sigma^n(a)=a$, so $\sigma^m(a)\in A$.    
    
    Suppose that the result holds for any all $k'$ where $1<k\leq k'$.
    We will show that the result holds for $k-1$.
    Since $1<k\leq m$, we can plug $k$ for $i$ in the equation (\ref{e4}) and gain that 
    \begin{align}\label{e00}
    \begin{split}
        0&=b_{k}+w_{k}=b_{k}+\sum_{j=k-1}^m\binom{j}{k-1}_qb_j\sigma^{k-1}D^{j-k+1}(a)\\
        &=b_{k}+\binom{k-1}{k-1}_qb_{k-1}\sigma^{k-1}D^{k-1-k+1}(a)+\sum_{j=k}^m\binom{j}{k-1}_qb_j\sigma^{k}D^{j-k+1}(a)\\
        &=b_{k}+b_{k-1}\sigma^{k-1}(a)+\sum_{j=k}^m\binom{j}{k-1}_qb_j\sigma^{k}D^{j-k+1}(a)
   \end{split}
    \end{align}
    By the inductive hypothesis, there are $r_{k'}$ in $A$ such that $b_{k'}r_{k'}\in N$ for $k'\geq k$.
    Multiply each side of the euqation (\ref{e00}) by $r_kr_{k+1}\dots r_m\in A$.
    \begin{align*}
        0&=b_{k}r_kr_{k+1}\dots r_m+b_{k-1}\sigma^{k-1}(a)r_kr_{k+1}\dots r_m\\
        &\hspace{2in }+\sum_{j=k}^m\binom{j}{k-1}_qb_j\sigma^{k}D^{j-k+1}(a)r_kr_{k+1}\dots r_m
    \end{align*}

    \noindent Since $a+N\in Z(R/N)$ and $\sigma$ is an automorphism, $\sigma(a)+N\in Z(R/N)$ and so $r+N\in Z(R/N)$ for all $r\in A$. 
    So,
      \begin{align*}
        0&=b_{k}r_kr_{k+1}\dots r_m+b_{k-1}\sigma^{k-1}(a)r_kr_{k+1}\dots r_m\\
        &\hspace{2in }+\sum_{j=k}^m\binom{j}{k-1}_qb_j\sigma^{k}D^{j-k+1}(a)r_kr_{k+1}\dots r_m\\
        &=b_{k}r_kr_{k+1}\dots r_m+b_{k-1}\sigma^{k-1}(a)r_kr_{k+1}\dots r_m\\
        &\hspace{1in }+\sum_{j=k}^m\binom{j}{k-1}_qb_jr_jr_kr_{k+1}\dots r_{j-1}r_{j+1}\dots r_m\sigma^{k}D^{j-k+1}(a)\hspace{0.4cm}(\text{mod}\hspace{0.2cm}N)\\ 
        &=b_{k-1}\sigma^{k-1}(a)r_kr_{k+1}\dots r_m \hspace{0.4cm}(\text{mod}\hspace{0.2cm}N).
    \end{align*}
    We have that $b_{k-1}\sigma^{k-1}(a)r_kr_{k+1}\dots r_m\in N$ and so $r_{k-1}:=\sigma^{k-1}(a)r_kr_{k+1}\dots r_m\in A$.

    By the claim, there exist $r_1, r_2,\dots,r_m\in A$ such that $b_ir_i\in N$ for $i=1,2,\dots,m$.
    Multiply equation (\ref{e5}) by $r_1r_{2}\dots r_m$.
    Since $r+N\in Z(R/N)$ for all $r\in A$,
    \begin{align*}
        0&=b_1r_1r_{2}\dots r_m+ar_1r_{2}\dots r_m+\sum_{j=1}^mb_jD^{j}(a)r_1r_{2}\dots r_m\\
        &=b_{1}r_1r_{2}\dots r_m+ar_1r_{2}\dots r_m+\sum_{j=1}^mb_jr_jr_1r_2\dots r_{j-1}r_{j+1}\dots r_mD^{j}(a)\hspace{0.4cm}(\text{mod}\hspace{0.2cm}N).\\
        &=ar_1r_{2}\dots r_m
        \hspace{0.4cm} (\text{mod}\hspace{0.2cm} N).
    \end{align*}

    \noindent Since $r_1r_{2}\dots r_m\in A$, we have that $r_1r_{2}\dots r_m+N=a^{t_0}(\sigma(a))^{t_1}\dots(\sigma^{n-1}(a))^{t_{n-1}}+N$ for some non-negative integers $t_0, t_1, \dots, t_{n-1}$.
    Since $ar_1r_{2}\dots r_m\in N$ and $r+N\in Z(R/N)$ for all $r\in A$,  we gain that $(a\sigma(a)\sigma^2(a)\dots \sigma^{n-1}(a))^t\in N$ for some $t\geq max\{t_0, t_1, \dots, t_{n-1}\}$. 
    By Lemma \ref{l12}, we know that $a\sigma(a)\sigma^2(a)\dots\sigma^{n-1}(a)\in N$.

    Since $\sigma$ is an isomorphism of $R$, we know that $\sigma(N)=N$.
    Since $\sigma^n(a+N)=a+N$, let $n_0$ be the smallest integer such that $\sigma^{n_0}(a+N)=a+N$.
    Then $n_0\leq n$.
    Since $a\sigma(a)\sigma^2(a)\dots\sigma^{n-1}(a)\in N$, so $(a+N)\sigma(a+N)\dots\sigma^{n-1}(a+N)\subseteq N$. 
    Since $r+N\in Z(R/N)$ for all $r\in A$, we have that $(a+N)^{j_0}\sigma(a+N)^{j_1}\dots\sigma^{n_0-1}(a+N)^{j_{n_0-1}}\subseteq N$. 
    Then $a^{j_0}\sigma(a)^{j_1}\dots\sigma^{n_0-1}(a)^{j_{n_0-1}}\in N$ for some integers $j_0, j_1,\dots,j_{n_0-1}$ and so $(a\sigma(a)\dots \sigma^{n_0-1}(a))^{t'}\in N$ for some integer $t'$.
    By Lemma \ref{l12}, $a\sigma(a)\dots \sigma^{n_0-1}(a)\in N$.
    Thus, $(a+N)\sigma(a+N)\dots \sigma^{n_0-1}(a+N)\subseteq N$.

    Let $c_0=a, c_1=\sigma(a), \dots, c_{n_0-1}=\sigma^{n_0-1}(a)$.
    Then $c_0c_1\dots c_{n_0-1}\in N$.
    We can find the largest integer $k$ such that such that there exists $i_1<i_2<\dots<i_k$ with $c_{i_1}c_{i_2}\dots c_{i_k}\notin N$ and $c_{i_1}c_{i_2}\dots c_{i_k}c_{i_{k+1}}\in N$ where $c_{i_1}, \dots, c_{i_k}, c_{i_{k+1}}$ are distinct. 
    Let $a_1=c_{i_1}c_{i_2}\dots c_{i_k}$.
    Since $\sigma^{n_0}(a+N)=a+N$, so $\sigma^{n_0}(a_1+N)=a_1+N$.
    Let $u$ be the smallest integer such that $\sigma^u(a_1+N)=a_1+N$.
    Then $u$ is the smallest integer such that 
    \begin{align*}
        \sigma^u((c_{i_1}+N)(c_{i_2}+N)\dots(c_{i_k}+N))=(c_{i_1}+N)(c_{i_2}+N)\dots(c_{i_k}+N).
    \end{align*}
    Since $c_{i_1}+N, \dots, c_{i_k}+N$ commute, if $\sigma^s$ permutes $c_{i_1}, c_{i_2},\dots,c_{i_k}$, then 
    \begin{align*}
        \sigma^s((c_{i_1}+N)(c_{i_2}+N)\dots(c_{i_k}+N))=(c_{i_1}+N)(c_{i_2}+N)\dots(c_{i_k}+N)
    \end{align*}
    and so by the minimality of $u$, $s\geq u$.
    Thus, if $s<u$, then $\sigma^s(a_1)$ involves some $c_{i_{k+1}}=\sigma^{i_{k+1}}(a)\notin \{ c_{i_1}, c_{i_2},\dots,c_{i_k} \}$.

    Let $a_2=a_1+\sigma(a_1)+\dots+\sigma^{u-1}(a_1)$.
    Then $a_2\in S$ and $a_2+N\in Z(R/N)$.
    If $a_2\in N$, then $a_1a_2=a_1^2 +a_1\sigma(a_1)+\dots+a_1\sigma^{u-1}(a_1)\in N$.
    By the previous paragraph, $\sigma^j(a_1)$  for $1\leq j\leq u-1$ contains some $c_{i_{k+1}}\notin \{ c_{i_1}, c_{i_2},\dots,c_{i_k} \}$.
    So, by the maximality of $k$, $a_1\sigma^j(a_1)\in N$ for $1\leq j\leq u-1$.
    Thus, $a_1^2\in N$ and by Lemma \ref{l12}, $a_1\in N$ which is a contradiction. 
    Thus, $a_2\notin N$. 
    Moreover, since $\sigma^u(a_1+N)=a_1+N$, we have that $\sigma(a_2+N)=a_2+N$.
    Since $a_2\notin N$, $a_2\in S$, and $a_2+N\in Z(R/N)$, we can proceed a similar way to the previous steps comparing coefficients of $g(x)+a_2x+g(x)a_2x=0$ where $g(x)$ is the quasi-inverse of $a_2x$.
    We gain that $(a_2\sigma(a_2)\dots \sigma^{q-1}(a_2))^p\in N$ for some positive integer $p$ and $q$ such that $\sigma^{q}(a_2)=a_2$. By Lemma \ref{l12}, $a_2\sigma(a_2)\dots \sigma^{q-1}(a_2)\in N$.
    Since $\sigma(a_2+N)=a_2+N$, $a_2^q\in N$. 
    By Lemma \ref{l12}, $a_2\in N$ which is a contradiction.
    Hence, $L=0$.
\end{proof}

\begin{theorem} \label{t4}
    Let $R$ be an algebra over a field of characteristic $p>0$, $\sigma$ be a locally torsion automorphism of $R$, and $D$ be a locally nilpotent $\sigma$-derivation on $R$.
    If $\sigma D=D\sigma$, then $J(R[x;\sigma,D])\cap R$ is a nil ideal of $R$.
\end{theorem}

\begin{proof}
    The approach follows \cite{P8}.
    We know that $J(R[x;\sigma,D])\cap R$ is an ideal of $R$. We will show that $J(R[x;\sigma,D])\cap R$ is nil.
    Take any element $a\in J(R[x;\sigma,D])\cap R$.
    Since $\sigma$ is locally torsion and $D$ is locally nilpotent, we can take an integer $ m>0$ such that $\sigma^{m}(a)=a$ and $D^{p^m}(a)=0$.
    By Proposition \ref{p2}, we have $x^{mp^m}a=ax^{mp^m}$.
    Let $R^*$ be a ring $R$ with adjoined identity.
    Since $a\in J(R[x;\sigma,D])\cap R\subseteq J(R[x;\sigma,D])$ and $J(R[x;\sigma,D])$ is an ideal of $R^*[x;\sigma,D]$, $ax^{mp^m}$ has the quasi-inverse $f(x)=\sum_{i=0}^tb_ix^i\in R[x;\sigma,D]$.
    Let $A$ be a subring of $R$ generated by $a$, $b_i$, $\sigma^l(D^j(a))$, $\sigma^l(D^j(b_i))$ for integers $0\leq i\leq t$ and $j\geq 0$ and $l\in \mathbb{Z}$.
    Since $\sigma D=D\sigma$, $\sigma$ is an automorphism of $A$ and $D$ is a $\sigma$-derivation on $A$, so $A[x;\sigma,D]$ is a subring of $R[x;\sigma,D]$. Note that $ax^{mp^m}$ is quasi-regular in $A[x;\sigma,D]$.
    
    Since $D$ is locally nilpotent and $\sigma$ is locally torsion, there is an integer $k=mq$ for some $q>0$ such that $\sigma^{k}(b_i)=b_i$ and $D^{p^k}(b_i)=0$ for all $i=0,1,2,\dots,t$.
    By Proposition \ref{p2}, $x^{kp^k}b_i=b_ix^{kp^k}$ for all $i=0,1,2,\dots,t$.
    Write $kp^k=k_1p^{k_2}$ for $(k_1,p)=1$. Then $k_2\geq k$.
    Since $\sigma D=D\sigma$, we have that
    \begin{align*}
        x^{kp^k}\sigma^l(D^j(b_i))-\sigma^l(D^j(b_i))x^{kp^k}
        &=x^{kp^k}\sigma^l(D^j(b_i))-D^j(\sigma^l(b_i))x^{kp^k}\\
        &=x^{kp^k}\sigma^l(D^j(b_i))-D^j(\sigma^{l+kp^k}(b_i))x^{kp^k}\\
        &=x^{kp^k}\sigma^l(D^j(b_i))-\sigma^{kp^k}(\sigma^l(D^j(b_i)))x^{kp^k}\\
        &=\sum_{s=0}^{kp^k-1}\binom{kp^k}{s}\sigma^sD^{kp^k-s}(\sigma^l(D^j(b_i)))x^s\\
        &=\sum_{s=0}^{kp^k-1}\binom{kp^k}{s}\sigma^{s+l}(D^{j+kp^k-s}(b_i))x^s\\
        &=\sum_{s=0}^{kp^k-1}\binom{kp^k}{kp^k-s}\sigma^{s+l}(D^{j+kp^k-s}(b_i))x^s.
    \end{align*}
    For $s=kp^k-1, kp^k-2, \dots, 1$, the binomial $\binom{kp^k}{kp^k-s}$ is divisible by $p$, and so $\binom{kp^k}{kp^k-s}\sigma^{s+l}(D^{j+kp^k-s}(b_i))=0$.
    Thus, $x^{kp^k}\sigma^l(D^j(b_i))-\sigma^l(D^j(b_i))x^{kp^k}=\sigma^l(D^{j+kp^k}(b_i))$.
    Since $D^{p^k}(b_i)=0$, $x^{kp^k}\sigma^l(D^j(b_i))-\sigma^l(D^j(b_i))x^{kp^k}=0$.
    Moreover, since $k=mq$, we can check that $\sigma^k(a)=a$, $D^{p^k}(a)=0$, and $x^{kp^k}a=ax^{kp^k}$. Similar to $x^{kp^k}\sigma^l(D^j(b_i))=\sigma^l(D^j(b_i))x^{kp^k}$, we can find that $x^{kp^k}\sigma^l(D^j(a))=\sigma^l(D^j(a))x^{kp^k}$.
    So, $x^{kp^k}$ commutes with all elements in $A$.

   Consider $A[x;\sigma,D]\subseteq A[[x;\sigma,D]]$ which is a power series ring in the indeterminate $x$.
    Recall that $x^{mp^m}a=ax^{mp^m}$, thus $(ax^{mp^m})^i=a^ix^{imp^m}$.
    In $A[[x;\sigma,D]]$, we can see that $q(x)=(-ax^{mp^m})+(-ax^{mp^m})^2+(-ax^{mp^m})^3+\dots$ is the quasi-inverse of $ax^{mp^m}$, that is,
    \begin{align*}
        ax^{mp^m} + q(x) + ax^{mp^m}q(x) =0.
    \end{align*}
    By the uniqueness of a quasi-inverse, $q(x)=f(x)\in A[x;\sigma,D]\subseteq A[[x;\sigma,D]]$ and so 
    \begin{align*}
        \sum_{i=0}^{\infty}(-ax^{mp^m})^i=\sum_{i=0}^{\infty}(-a)^ix^{imp^m}\in A[x;\sigma,D]
    \end{align*}
    Hence, $a^i=0$ for some $i$.
\end{proof}

\section{Acknowledgements}

The author is grateful to Professor Mikhail Chebotar for his assistance and guidance. The author would also like to thank the referee for their careful reading of the manuscript and for their valuable comments and suggestions, which have improved the paper.

\adresa
Jooyoung Shin, Sam Houston State University, 77340, Huntsville, USA
e-mail: jxs361@shsu.edu


{\small
\begin{thebibliography}{999}
\bibitem{P7} S. A. Amitsur, \textit{ Radicals of polynomial rings}, Canadian J. Math., \textbf{8} (1956), 355-361.

\bibitem{P13} S. S. Bedi and J. Ram, \textit{Jacobson radical of skew polynomial rings and skew group rings}, Israel J. Math., \textbf{35} (1980), no.4, 327–338.

\bibitem{bg} J. Bergen and P. Grzeszczuk, \textit{Skew power series rings of derivation type}, J. Algebra Appl., \textbf{10} (2011), no.6, 1383-1399.

\bibitem{bo} N. Bourbaki, \textit{Commutative algebra}, Addison-Wesley Publishing Company, Massachusetts, 1972.

\bibitem{ct} C. L. Chuang and Y. T. Tsai, \textit{Higher derivations of Ore extensions by q-skew derivations}, AJournal of Pure and Applied Algebra, \textbf{214} (2010), 1778–1786.

\bibitem{P5} M. Ferrero, K. Kishimoto, and K. Motose, \textit{On radicals of skew polynomial rings of derivation type}, J. London Math. Soc., (2)\textbf{28} (1983), no.1, 8-16.

\bibitem{P2} B. Greenfeld, A. Smoktunowicz, and M. Ziembowski, \textit{Five solved problems on radicals of Ore extensions}, Publ. Mat., \textbf{63} (2019), no 2, 423-444.



\bibitem{P15} P. Grzeszczuk, A. Leroy and J. Matczuk, \textit{Artinian property of constants of algebraic q-skew derivations}, Israel J. Math. \textbf{121} (2001), 265-284.

\bibitem{ko} G. Köthe, \textit{Die Struktur der Ringe, deren Restklassenring nach dem Radikal vollständig reduzibel ist}, Math. Z., \textbf{32} (1930) no.1, 161-186.

\bibitem{P11} J. Krempa, \textit{Logical connections between some open problems concerning nil rings}, Fund. Math, \textbf{76} (1972), 121-130.


\bibitem{P4} B. W. Madill, \textit{On the Jacobson radical of skew polynomial extensions of rings satisfying a polynomial identity}, Comm. Algebra, \textbf{44} (2016), no.3, 913-918.


\bibitem{P1} L. H. Rowen, \textit{Polynomial identities in ring theory}, Pure and Applied Mathematics, vol. 84, Academic Press, Inc., Harcourt Brace Jovanovich, Publishers, New York-London, 1980.

\bibitem{P8} A.  Smoktunowicz, \textit{How far can we go with Amitsur's theorem in differential polynomial rings?}, Israel J. Math., \textbf{219} (2017), no.2, 555-608.

\end{thebibliography}}
\end{document}